\documentclass[12pt,a4paper, reqno]{article}

\usepackage{indentfirst} 
\usepackage{amsmath, amsthm, amssymb,xfrac} 
\usepackage{tikz} 
\usetikzlibrary{patterns}
\usepackage{wrapfig} 
\usepackage{xcolor} 
\usepackage{verbatim} 
\usepackage{enumerate} 
\usepackage{todonotes} 

\usepackage[section]{algorithm} 
\usepackage{algpseudocode}

\usepackage{hyperref} 
\hypersetup{
colorlinks=true,
linkcolor=black!50!red,
linktoc=all
}
\usepackage[all]{hypcap} 

\newcounter{constant}
\newcommand{\newconstant}[1]{\refstepcounter{constant}\label{#1}}
\newcommand{\useconstant}[1]{c_{\textnormal{\tiny~\ref{#1}}}}
\newcounter{bigconstant}

\newtheorem{theorem}{Theorem}[section]
\newtheorem{proposition}[theorem]{Proposition}
\newtheorem{lemma}[theorem]{Lemma}

\theoremstyle{definition}

\newtheorem{remark}[theorem]{Remark}

\numberwithin{equation}{section}

\newcommand{\PP}{\mathbb{P}}
\newcommand{\EE}{\mathbb{E}}
\newcommand{\RR}{\mathbb{R}}

\newcommand{\NN}{\mathbb{N}}
\newcommand{\ZZ}{\mathbb{Z}}

\newcommand{\charf}[1]{\mathbf{1}_{#1}}

\DeclareMathOperator{\dist}{d}

\DeclareMathOperator{\var}{Var}

\DeclareMathOperator{\Inf}{Inf}
\DeclareMathOperator{\Arm}{Arm}

\begin{document}

\title{Noise sensitivity and Voronoi percolation}

\author{Daniel Ahlberg\footnote{Department of Mathematics, Stockholm University, SE - 106 91 Stockholm, Sweden.} \and Rangel Baldasso\footnote{IMPA, Estrada Dona Castorina 110, 22460-320, Rio de Janeiro, RJ, Brazil.}
{\let\thefootnote\relax\footnote{E-mail: dahlb@math.su.se; baldasso@impa.br}}
}

\maketitle

\begin{abstract}
In this paper we study noise sensitivity and threshold phenomena for Poisson Voronoi percolation on $\mathbb{R}^2$. In the setting of Boolean functions, both threshold phenomena and noise sensitivity can be understood via the study of randomized algorithms. Together with a simple discretization argument, such techniques apply also to the continuum setting.
Via the study of a suitable algorithm we show that box-crossing events in Voronoi percolation are noise sensitive and present a threshold phenomenon with polynomial window. We also study the effect of other kinds of perturbations, and emphasize the fact that the techniques we use apply for a broad range of models.
\end{abstract}

\section{Introduction}\label{sec:intro}

\par The concept of a Boolean function, $f:\{0,1\}^n\to\{0,1\}$, is of fundamental importance in theoretical computer science. Moreover, many of the most well-studied problems in the intersection between combinatorics and probability theory may be phrased in terms of (often monotone) Boolean functions. One is, in this context, interested in the typical behaviour of a Boolean function for an element in $\{0,1\}^n$ chosen according to product measure with marginal density $p$, henceforth denoted by $\PP_p$.
The study of Boolean functions has led to a vast literature on a range of fascinating phenomena, such as the existence of thresholds and the effect of small perturbations, see e.g.~\cite{book_ns,odonnell}.

\par Threshold phenomena of monotone Boolean functions were first discovered by Erd\H os and R\'enyi~\cite{er} in their pioneering study of random graphs. The existence of a sharp threshold is the essence of Kesten's celebrated 1980 proof that the critical probability for the existence of an infinite connected component in bond percolation on $\mathbb{Z}^2$ equals $\sfrac12$~\cite{kesten80}.
A sequence $(f_n)_{n\ge1}$ of monotone\footnote{A Boolean function is monotone if $f_n(\omega')\ge f_n(\omega)$ whenever $\omega'\ge\omega$ coordinate-wise.} Boolean functions $f_n:\{0,1\}^n\to\{0,1\}$ is said to have a {\bf threshold} at $p\in(0,1)$ if, for every $\epsilon>0$, we have 
$$
\lim_{n\to\infty}\PP_{p-\epsilon}[f_n=1]=0\quad\text{and}\quad\lim_{n\to\infty}\PP_{p+\epsilon}[f_n=1]=1.
$$
The understanding of thresholds has increased with works by Russo~\cite{russo82}, Kahn, Kalai and Linial~\cite{kkl}, Friedgut and Kalai~\cite{fk}, and Talagrand~\cite{tal}.

\par The notion of noise sensitivity was introduced in a seminal paper by Benjamini, Kalai and Schramm~\cite{bks}. Given $\omega\in\{0,1\}^n$, chosen according to $\PP_p$, we obtain an $\epsilon$-perturbation $\omega^\epsilon$ of $\omega$ by resampling each bit of $\omega$ independently with probability $\epsilon$. A sequence $(f_n)_{n\ge1}$ of functions $f_n:\{0,1\}^n\to\{0,1\}$ is said to be {\bf noise sensitive} at level $p$ (NS$_p$ for short) if $f_n(\omega)$ and $f_n(\omega^\epsilon)$ are asymptotically uncorrelated, i.e., if
\begin{equation}\label{eq:ns_def}
\EE_p[f_n(\omega)f_n(\omega^\epsilon)]-\EE_p[f_n(\omega)]^2\to0, \quad\text{as }n\to\infty.
\end{equation}
The study of noise sensitivity has led to a detailed understanding of certain planar percolation models, both discrete: Benjamini, Kalai and Schramm~\cite{bks}, Schramm and Steif~\cite{ss}, Garban, Pete and Schramm~\cite{gps}, and in the continuum: Ahlberg, Broman, Griffiths and Morris~\cite{abgm}, and Ahlberg, Griffiths, Morris and Tassion~\cite{qvp}.

\par In this paper we study threshold phenomena and the effect of small perturbations in the context of Poisson Voronoi percolation on $\mathbb{R}^2$. Our contributions in this direction are two-fold. First, we describe the discretization method developed in~\cite{abgm}, by which we reduce the continuum problem to its discrete counterpart, and emphasize the close relation between threshold phenomena and noise sensitivity of Boolean functions via the study of randomized algorithms. Combining the two techniques we derive quantitative estimates on the width of the threshold window and the rate of decorrelation in~\eqref{eq:ns_def}.
Second, we discuss a range of different but related notions of perturbations in the context of Voronoi percolation.
Some of these notions we examine in detail, whereas other are left as open problems.

\par We remark that the application of the discretization approach is here somewhat simpler than as originally developed in~\cite{abgm}. Moreover, the techniques we use apply to a range of continuum percolation models such as Poisson Boolean percolation and confetti percolation, as opposed to the approach in~\cite{qvp} that exploits colour-switching tricks. For self-dual models, such as Voronoi and confetti percolation, our approach offers an alternative proof that the critical probability for percolation equals $\sfrac12$, as originally proved by Bollob\'as and Riordan~\cite{br}. In addition, the quantitative estimates that we obtain on the size of the threshold window are new. We have chosen to present our results in terms of Voronoi percolation as this model offers a range of possibilities when it comes to different perturbations.

\bigskip

\par {\bf Description of Voronoi percolation.} Poisson Voronoi percolation is a model for the study of long-range connections in a two-colouring of $\RR^2$ based on a tessellation. The large-scale behaviour in models of this kind is well-known to be governed by its behaviour in finite regions, and we shall for this reason work with the restriction of the model to the unit square.
Let, hence, $S:=[0,1]^{2}$ and let $\Omega$ denote the space of finite subsets of $S\times\{0,1\}$, equipped with the Borel sigma algebra. Formally we construct a Voronoi configuration on $S$ based on a Poisson point process $\eta$ on $\Omega$ with intensity measure $n\lambda_S\otimes[p\delta_1+(1-p)\delta_0]$, where $\lambda_S$ denotes Lebesgue measure on $S$.

\par Given $\eta\in\Omega$, we define the Voronoi cell associated to $(x,u)\in\eta$ as
$$
V(x):=\big\{y\in S:\dist(y,x)\le\dist(y,x')\text{ for all }(x',u')\in\eta\big\},
$$
where $\dist$ denotes the Euclidean distance. 
Based on the tessellation we declare a point in $S$ \emph{red} or \emph{blue} depending on whether it is contained in the cell corresponding to a point in $\eta$ with $u$-coordinate 0 or 1, respectively.\footnote{It is not hard to see that, with probability one, every Voronoi cell is a closed bounded convex set. A point on the boundary of some set may belong to more than one cell, but no point of $S$ can belong to more than three cells. Besides, if two cells share a vertex, they share an entire edge. We can therefore ignore the fact that points on the boundary of two cells may be declared both red and blue.}
To rule out degenerate cases, we colour all points in $S$ red in the case that $\eta=\emptyset$.
We shall denote the associated measure by $\PP_{n,p}$, and we will occasionally suppress the subscript to ease the notation.

\newconstant{c:non_trivial}
\par Given a rectangle $R\subseteq S$, let $H_R$ denote the event defined by the existence of a continuous blue path crossing $R$ horizontally, and let $f_R:\Omega\to\{0,1\}$ denote the indicator of the event $H_R$. Conditioned on $\eta \neq \emptyset$, at $p=\sfrac{1}{2}$ the model is self-dual, meaning that the red and blue components are equi-distributed. Since any rectangle $R\subseteq S$ is either crossed horizontally by a blue path or vertically by a red path, it follows by symmetry that\footnote{Equality would here hold would it not be for the possibility that $\eta$ may be empty.}
$$
\PP_{n,\sfrac{1}{2}}[f_S=1]\to\sfrac{1}{2}.
$$
Indeed, the function $f_R$ is non-degenerate at $p=\sfrac12$ for any rectangle $R\subseteq S$: There exists a constant $\useconstant{c:non_trivial}>0$, depending only on the aspect ratio of $R$, such that
\begin{equation}\label{eq:non_trivial}
\useconstant{c:non_trivial} \leq \PP_{n,\sfrac{1}{2}}[f_{R}=1] \leq 1-\useconstant{c:non_trivial},
\end{equation}
uniformly in $n$. This was first proved by Tassion~\cite{tassion} for Voronoi percolation on $\RR^2$, and later extended in~\cite{qvp} to subsets of $\RR^2$ with boundary.
The box-crossing property in~\eqref{eq:non_trivial} is a typical critical phenomenon and a suggestive indication that the critical threshold for the existence of an unbounded connected blue component in Poisson Voronoi percolation on $\RR^2$ equals $\sfrac{1}{2}$.

\bigskip

\par {\bf Description of results.} In the continuum setting, a natural notion of perturbation of a Voronoi configuration is obtained as follows. For $\epsilon\in(0,1)$ let $\eta(\epsilon)$ be obtained from $\eta$ by first thinning $\eta$ by a factor $1-\epsilon$ and then sprinkling an independent density of $\epsilon n$ points to regain the initial density $n$. The collection of blue points in each $\eta$ and $\eta(\epsilon)$ is distributed as a Poisson point process of intensity $pn$. We shall say that the function $f_R:\Omega\to\{0,1\}$, encoding the existence of a horizontal blue crossing of the rectangle $R$, is {\bf noise sensitive} at level $p$ if, for every $\epsilon>0$, we have
\begin{equation}\label{eq:ns_voronoi}
\EE_{n,p}\big[f_{R}(\eta)f_{R}(\eta(\epsilon))\big]-\EE_{n,p}\big[f_{R}(\eta)\big]^{2} \to 0,\quad \text{as }n\to\infty.
\end{equation}
Moreover, we say that $f_R$ has {\bf positive noise sensitivity exponent} if~\eqref{eq:ns_voronoi} holds with $\epsilon$ replaced by $\epsilon_n=n^{-\alpha}$ for some $\alpha>0$.

\par Notice that in~\eqref{eq:ns_voronoi} we have defined what it means for a single function to be noise sensitive. The definition is nevertheless analogous to the one in~\eqref{eq:ns_def}, where a sequence of functions was considered.

\par Our first theorem states that box crossings in Poisson Voronoi percolation are noise sensitive at the critical parameter $p=\sfrac12$, and that the associated noise sensitivity exponent is positive.

\begin{theorem}\label{teo:ns}
For every rectangle $R \subseteq S$, the function $f_{R}$ is noise sensitive at level $p=\sfrac12$ with a positive noise sensitivity exponent.
\end{theorem}


\par Our second result concerns the width of the threshold (or critical) window, as a function of $n$, in which the probability of a horizontal blue crossing is bounded away from zero and one. That the width tends to zero with $n$ is the essence of Bollob\'as and Riordan's proof that the critical probability for Poisson Voronoi percolation on $\RR^2$ equals $\sfrac{1}{2}$. We show that the width of the critical window tends to zero polynomially in $n$, and hence provide an alternative proof of Bollob\'as and Riordan's theorem.

\begin{theorem}\label{teo:tw}
For every rectangle $R \subseteq S$ there exists $\gamma>0$ such that
\begin{equation*}
\lim_{n\to\infty}\PP_{n,\sfrac{1}{2}-n^{-\gamma}}[f_{R}=1]= 0
\quad\text{and}\quad
\lim_{n\to\infty}\PP_{n,\sfrac{1}{2}+n^{-\gamma}}[f_{R}=1]= 1.
\end{equation*}
\end{theorem}

We remark that the width of the critical window cannot decay faster than order $\sfrac{1}{\sqrt{n}}$ due to the well-known fact that no sequence of monotone Boolean functions may have a smaller threshold window; see e.g.~\cite{ahlste}. We mention that in parallel work, Duminil-Copin, Raoufi and Tassion~\cite{dumcopraotas2} present yet another proof of the Bollob\'as-Riordan theorem. We also mention that the existence of a threshold at $p=\sfrac{1}{2}$, together with Cauchy-Schwarz' inequality, implies that Voronoi percolation is trivially noise sensitive for $p\neq\sfrac{1}{2}$.

\par One way to think of the perturbation in~\eqref{eq:ns_voronoi} is as the following dynamical process evolving in time: Let points appear in $S\times\{0,1\}$ at rate $n$, where they remain for an exponentially distributed time before disappearing. The measure $\PP_{n,\sfrac12}$ is stationary for this process, and for $\epsilon=1-e^{-t}$ the pair $(\eta,\eta(\epsilon))$ corresponds to the dynamical process observed at times $0$ and $t$.

\par In greater generality we may think of a perturbation as a reversible time-homogeneous Markov process $(\eta(t))_{t\ge0}$ on $\Omega$ evolving in equilibrium. For each such process, the Markov property and reversibility together give that
\begin{equation*}
\begin{aligned}
&\EE\big[f_R(\eta(0))f_R(\eta(t))\big]-\EE\big[f_R(\eta(0))\big]^2\\
&\quad=\,\EE\Big[\EE\big[f_R(\eta(0))\big|\eta(\sfrac t2)\big]\EE\big[f_R(\eta(t))\big|\eta(\sfrac t2)\big]\Big]-\EE\big[f_R(\eta(0))\big]^2\\
&\quad=\,\var\Big(\EE\big[f_R(\eta(\sfrac t2))\big|\eta(0)\big]\Big).
\end{aligned}
\end{equation*}
Hence, for each dynamical process of this kind, the correlation between two points in time measures the amount of information in some sigma algebra $\mathcal{F}$ -- the sigma algebra generated by the glimpse of the process in one of the time points -- and being sensitive with respect to this information is equivalent to
\begin{equation}\label{eq:general}
\var_{n,\sfrac12}\big(\EE[f_R(\eta)|\mathcal{F}]\big)\to0,\quad\text{as }n\to\infty.
\end{equation}
Clearly, the more information contained in $\mathcal{F}$ the larger the variance. This indicates, in particular, that more conservative dynamics tend to affect a system to a lesser extent. Two natural notions of perturbations that conserve the number of points are
\begin{itemize}\itemsep1pt \parskip0pt \parsep0pt
\item re-randomize colours of a small proportion of points;
\item re-randomize locations of a small proportion of points.
\end{itemize}
The former of these two notions was studied in~\cite{qvp}, where the authors showed that the existence of crossings in Voronoi percolation are sensitive with respect to resampling a small proportion of the colours. The latter we study in this paper, and show that Voronoi crossings are sensitive also with respect to relocation of points within $S$.

\begin{theorem}\label{teo:other_noises}
Let $\eta^*$ be obtained from $\eta$ by re-randomizing the location of each point in $\eta$ independently and uniformly within $S$ with probability $\epsilon>0$. For every $\epsilon>0$ and rectangle $R\subseteq S$, we have
$$
\EE_{n,\sfrac12}\big[f_R(\eta)f_R(\eta^*)\big]-\EE_{n,\sfrac12}\big[f_R(\eta)\big]^2\to0,\quad\text{as }n\to\infty.
$$
\end{theorem}

We remark that the statement of the above theorem remains true for $\epsilon$ replaced by $\epsilon_n=n^{-\alpha}$ for some $\alpha>0$, which is a direct consequence of $f_R$ having a positive noise sensitive exponent (see Theorem~\ref{teo:ns}).

\bigskip

\par \textbf{Open problems.} The discussion above, in which a perturbation was described in terms of a reversible Markov process evolving in equilibrium, suggests that there is a whole range of possible perturbations, apart from those considered above. Different notions of perturbations are likely to require new ideas and additional techniques than those explored here. One open problem, for which the techniques of this study are insufficient, is to determine whether percolation crossings are sensitive with respect to relocation of (a proportion of) points of a given colour, while points of the other colour are kept fixed. We believe that this is the case, and motivate our belief with the fact that crossings in Bernoulli percolation on $\ZZ^2$ are sensitive to resampling of vertical bonds, while horizontal bonds are kept fixed; see~\cite{gps}.

\par Another notion of perturbation is obtained by perturbing the location of each point according to independent Brownian motions run for some time $t=t(n)$. Running the Brownian motions for time $\sfrac1n$ displaces a typical point on the order $\sfrac{1}{\sqrt{n}}$, which is the typical distance between two neighbouring points. At this time scale the perturbation is thus much more local than when completely resampling the locations. We conjecture that for $t(n)=\sfrac{\epsilon}{n}$ the points are perturbed enough to lose (in the sense of~\eqref{eq:general}) the information carried by the initial configuration regarding the existence of a horizontal blue crossing. An analogous statement has been conjectured to hold in the discrete setting, but remains unproven~\cite{exclusion_sensitivity}.

\par The perturbation considered in Theorem~\ref{teo:other_noises} could of course be made more local by relocating each point, not uniformly over the whole square, but uniformly over some smaller square of side length $n^{-\alpha}$, where $\alpha\in(0,1/2]$. It is possible that the argument used to prove Theorem~\ref{teo:other_noises} could be adapted in the case that $\alpha>0$ is sufficiently small, but coming down to the scale $\sfrac{1}{\sqrt{n}}$ will presumably require significant new ideas. Again, see~\cite{exclusion_sensitivity} for related work in a discrete setting.

\par In contrast to these open problems, we mention that Benjamini and Schramm~\cite{bensch98} have proved that Voronoi percolation in two and three dimensions is stable with respect to a (non-random) conformal perturbation of the underlying Euclidean metric that defines the tessellation. The result of Benjamini and Schramm is related to conformal invariance, which is believed to hold for Voronoi percolation, just as for many other planar percolation models. This is one important missing piece that remains in order to determine critical exponents by means of SLE technology. With such techniques at hand, one would be able to determine the width of the critical window and the (optimal) noise sensitivity exponent precisely. The latter should further have implications for the existence of exceptional times in certain dynamical versions of the processes considered here. We refer the reader to~\cite{gps,book_ns} for further discussion in this direction.

\bigskip

\par \textbf{Proof overview.} We will follow the approach developed in~\cite{abgm}, and revisited in~\cite{att}, by which the continuum problem is reduced to its discrete counterpart via a two-stage construction. The central idea is to consider a Poisson point process $\eta_k$ on $\Omega$ chosen according to $\PP_{kn,p}$ for some $k\ge1$, and obtain a configuration $\eta$ from $\eta_k$ via thinning. Conditional on $\eta_k$, we may think of $\eta$ as an element in, and $f_R$ as a function on, $\{0,1\}^{\eta_k}$. Conditional on $\eta_k$, we will be able to study the behaviour of $f_R$ via techniques developed for the analysis of Boolean functions.

\par Russo's approximate 0-1 law says that any sequence of monotone Boolean functions for which the influence of each bit tends to zero exhibits a threshold behaviour~\cite{russo82}. A more modern approach to threshold phenomena comes from randomized algorithms via the OSSS inequality~\cite{osss}. That randomized algorithms can be used to study threshold phenomena has previously been observed by Gady Kozma (see the appendix of~\cite{ahlste}) and in recent work by Duminil-Copin, Raoufi and Tassion~\cite{dumcopraotas2,dumcopraotas1}.
Randomized algorithms are also connected to noise sensitivity via the Schramm-Steif revealment theorem~\cite{ss}. In order to prove Theorems~\ref{teo:ns} and~\ref{teo:tw} we shall thus devise an algorithm that, conditional on $\eta_k$, queries points in $\eta_k$ sequentially until the outcome of $f_R(\eta)$ is determined. If, with high probability, the algorithm has low revealment, that is, is unlikely to query any specific point in $\eta_k$, then the results will follow.

\par The proof of Theorem~\ref{teo:other_noises} will also rely on the reduction to a discrete setting.
The dynamical process studied there is conservative, and in that sense related to the concept of exclusion sensitivity studied by Broman, Garban and Steif~\cite{exclusion_sensitivity}. We shall follow their approach, and instead of a direct study of the conservative dynamics, we shall show that there is a coupling between $(\eta,\eta(\epsilon))$ and $(\eta,\eta^*)$ such that $(f_R(\eta),f_R(\eta(\epsilon)))$ and $(f_R(\eta),f_R(\eta^*))$ agree with high probability. This will be possible due to a result in~\cite{exclusion_sensitivity} which says that any noise sensitive sequence of Boolean functions $(f_n)_{n\ge1}$ is unlikely to change when resampling up to order $\sqrt{n}$ of the variables. The result then follows by Theorem~\ref{teo:ns} and the observation that
\begin{equation*}
\Big|\EE_{n,p}\big[f_{R}(\eta)f_{R}(\eta(\epsilon))\big]-\EE_{n,p}\big[f_{R}(\eta)f_{R}(\eta^*)\big]\Big|  \,\leq\, \PP_{n,p}\big[ f_{R}(\eta(\epsilon)) \neq f_{R}(\eta^*)\big].
\end{equation*}

\bigskip

\textbf{Structure of the paper.} Tools and techniques from the analysis of Boolean functions will be central in the remainder of this paper. We shall in Section~\ref{sec:ds} begin with a brief review of these, centering on the use of randomized algorithms and their revealment. In Section~\ref{sec:cont_to_disc} we outline the discretization method developed in~\cite{abgm}, which will allow for these techniques to be applied in the setting of Voronoi percolation.
In Section~\ref{sec:algorithms} we describe an algorithm that will be used to prove Theorems~\ref{teo:ns} and~\ref{teo:tw}, and estimate its revealment. The proofs of Theorems~\ref{teo:ns} and~\ref{teo:tw} are then given in Section~\ref{sec:nstw}, and Sections~\ref{sec:srt} and~\ref{sec:fp} are dedicated to study the effect of alternative perturbations, and to prove Theorem~\ref{teo:other_noises}.

\bigskip

\textbf{Acknowledgements.} The authors thank Augusto Teixeira for valuable discussions. DA thanks Swedish Research Council for financial support through grant 637-2013-7302. RB thanks FAPERJ for financial support through grant E-26/202.231/2015.

\section{Analysis of Boolean functions}\label{sec:ds}

\par In the analysis of Boolean functions, discrete Fourier techniques have become an indispensable tool. Although phenomena such as sharp thresholds and noise sensitivity can be directly linked to the spectrum of the Fourier-Walsh decomposition of a Boolean function, it is often a very challenging task to obtain precise estimates on the spectrum itself. A range of techniques have therefore been developed in order to relate such phenomena to notions such as influence of variables and revealment of algorithms, which are typically more tractable quantities to estimate.

\par In this section, we review some results connecting influences and revealment to threshold behaviour and noise sensitivity. We shall avoid the discussion of Fourier techniques, that lie behind several of the results we describe, and refer the reader to the books~\cite{book_ns} and~\cite{odonnell} for a more extensive treatment.

\subsection{Influence of variables}\label{subsec:influences}

\par Let $[n]:=\{1,2,\ldots,n\}$ and $f:\{0,1\}^n\to\{0,1\}$ be a Boolean function. The {\bf influence} of bit $k\in[n]$ for $f$ is defined as
\begin{equation}\label{eq:influence}
\Inf_{k}^{p}(f)=\Inf_k^p(f,[n]):=\PP_{p}[f(\omega) \neq f(\sigma_{k}\omega)],
\end{equation}
where $\sigma_{k}$ is the operator that changes $\omega$ at position $k$ from $\omega_k$ to $1-\omega_k$. Recall that a Boolean function is called monotone if $f(\omega')\ge f(\omega)$ whenever $\omega'_k\ge\omega_k$ for each $k\in[n]$. It is well-known that many monotone Boolean functions exhibit a threshold phenomenon, where the probability $\PP_p[f=1]$ increases from close to 0 to close to 1 in a narrow window -- the threshold window. 
The central role of influences in the understanding of this phenomenon is emphasized by the Margulis-Russo formula. It says that, for any monotone function $f:\{0,1\}^n\to\{0,1\}$,
\begin{equation}\label{eq:russo_formula}
\frac{d}{dp}\PP_{p}[f=1]=\sum_{k=1}^{n}\Inf_{k}^{p}(f).
\end{equation}

\par
Russo's approximate 0-1 law~\cite{russo82} gives the first general condition for the existence of a threshold. Russo showed that for every $\epsilon>0$ there exists $\delta>0$ such that if $\Inf_k^p(f)\le\delta$ uniformly in $k$ and $p$, then $\PP_p(f=1)$ transitions from below $\epsilon$ to above $1-\epsilon$ in a window of width at most $\epsilon$. Later works~\cite{kkl,fk,tal} have obtained a more precise formulation of Russo's theorem that allows one to get a quantitative bound on the width of the threshold window.

\par Influences are likewise fundamentally connected to the notion of noise sensitivity. The BKS Theorem, due to Benjamini, Kalai and Schramm~\cite{bks}, says that a sufficient condition for a sequence $(f_n)_{n\ge1}$ of Boolean functions to be noise sensitive at level $p$ is that
\begin{equation}\label{eq:inf_square}
\sum_{k=1}^n\Inf_k^p(f_n)^2\to0\quad\text{as }n\to\infty.
\end{equation}
For monotone functions this condition is also necessary.

\subsection{Revealment of algorithms}\label{subsec:revealment}

\par A (randomized) algorithm is a rule which queries a subset of the bits of $\omega\in\{0,1\}^n$ in a random order, which is allowed to depend on what has been seen so far, and outputs either 0 or 1. An algorithm is said to {\bf determine} $f$ if its output equals $f(\omega)$ for each $\omega\in\{0,1\}^n$. The {\bf revealment} of an algorithm $\mathcal{A}$ with respect to $K \subseteq [n]$ is defined as
\begin{equation}\label{eq:revelament}
\delta_{p}(\mathcal{A},K):=\max_{k \in K}\PP_{p}[\mathcal{A} \text{ queries bit } k].
\end{equation}

\par In order to verify the condition in~\eqref{eq:inf_square}, Benjamini, Kalai and Schramm~\cite{bks} devised a method involving algorithms. This method was developed further in later work by Schramm and Steif~\cite{ss}. In essence, this method shows that a sequence of functions is noise sensitive if there exists (a sequence of) algorithms that determines $f_n$ without being likely to query any specific bit. The next proposition, due to Schramm and Steif~\cite{ss}, gives an explicit formulation of this last statement.

\begin{proposition}\label{prop:ssrt}
Let $\mathcal{A}$ be an algorithm that determines the function $f: \{0,1\}^{n} \to \{0,1\}$. Then, for every $p\in(0,1)$ and $m\ge1$, we have
\begin{equation*}
\EE_{p}[f(\omega)f(\omega^{\epsilon})]-\EE_{p}[f(\omega)]^{2} \leq e^{-\epsilon m}+m^{2}\delta_{p}(\mathcal{A}, [n]).
\end{equation*}
\end{proposition}

\par Since the correlation is non-negative, it is immediate from the proposition above that a sequence $(f_n)_{n\ge1}$ is noise sensitive if there exists an algorithm $\mathcal{A}$ determining $f_n$ with revealment tending to zero. Moreover, if $\delta_p(\mathcal{A},[n])$ decays polynomially fast, then the sequence $(f_n)_{n\ge1}$ has positive noise sensitivity exponent.

\par Randomized algorithms have also been related to influences and threshold phenomena via the following inequality, due to O'Donnell, Saks, Schramm and Servedio~\cite{osss}. 

\begin{proposition}\label{prop:osss}
Let $f:\{0,1\}^{n} \to \{0,1\}$ be a Boolean function and $\mathcal{A}$ an algorithm that determines $f$. Then, for every $p\in(0,1)$, we have
\begin{equation}
\var_{p}(f) \leq p(1-p)\sum_{k \in [n]} \delta_{p}(\mathcal{A},k)\Inf_{k}^{p}(f).
\end{equation}
\end{proposition}

The above inequality implies, in particular, that
\begin{equation*}
\var_{p}(f) \leq \frac{1}{4}\delta_{p}(\mathcal{A},[n])\sum_{k=1}^{n}\Inf_{k}^{p}(f),
\end{equation*}
and hence, together with the Margulis-Russo formula, one concludes that monotone Boolean functions satisfy the inequality
\begin{equation*}
\frac{d}{dp}\PP_{p}[f=1] \geq 4\frac{\var_{p}(f)}{\delta_{p}(\mathcal{A},[n])}.
\end{equation*}
As we shall see, we will, via the study of algorithms, be able to obtain polynomial bounds on the width of the threshold window of certain Boolean functions, where methods based on influences would give logarithmic bounds, see e.g.~\cite{fk}.

\par Although we shall not make use of this below, we mention the following upper bound on the sum of influences in terms of the revealment, due to O'Donnell and Servedio~\cite{odoser07}: For every function $f:\{0,1\}^n\to\{0,1\}$ that is monotone in each coordinate we have that
\begin{equation}\label{eq:influence_upper_bound}
\sum_{k\in[n]}\Inf_k^p(f)\,\le\,\sqrt{n\sum_{k\in[n]}\Inf_k^{p}(f)^2}\,\le\,\frac{1}{p(1-p)}\sqrt{n\,\delta_p(\mathcal{A},[n])}.
\end{equation}
The former of the two inequalities is immediate from Cauchy-Schwarz' inequality, whereas the latter follows from (a variant of) the Schramm-Steif revealment theorem. (See also~\cite[Theorem~VIII.8]{book_ns} for a direct proof.)
This inequality provides a way to obtain a lower bound, as opposed to the upper bound obtained via the OSSS inequality, on the width of the threshold window for monotone Boolean functions that is sharper than the elementary lower bound of order $\sfrac{1}{\sqrt{n}}$. We are not aware of any such application having previously appeared in the literature. However, see~\cite[Section~VIII.5]{book_ns} for an application showing that the critical four-arm exponent for Bernoulli percolation on $\ZZ^2$ is strictly larger than one.

\section{Continuum to discrete}\label{sec:cont_to_disc}

\par We now begin to set the stage for the proofs of Theorems~\ref{teo:ns} and~\ref{teo:tw}. Our approach will be based on a method developed in~\cite{abgm}, and revisited in~\cite{att}, that allows one to reduce the continuum problem at hand to its discrete counterpart via a two-stage construction of the continuum process.

\par Fix an integer $k \geq 2$ and choose $\eta_{k}\in\Omega$ distributed as $\PP_{kn, p}$. Let $\eta$ be obtained from $\eta_{k}$ by independently including each point of $\eta_{k}$ with probability $\sfrac{1}{k}$. Notice that $\eta$ is distributed according to $\PP_{n,p}$, and that conditional on $\eta_{k}$, we may consider $\eta$ as an element in $\{0,1\}^{\eta_{k}}$ chosen according to $\PP_{\sfrac{1}{k}}$.

\par Recall the notation $(\eta,\eta(\epsilon))$ for a pair of configurations in $\Omega$ distributed according to $\PP_{n,p}$, where the latter is an $\epsilon$-perturbation of the former. The two-stage construction gives an alternative way to obtain a pair of configurations $(\eta,\eta^\epsilon)$ where, conditional on $\eta_k$, the latter is obtained by an $\epsilon$-perturbation of the former seen as elements in $\{0,1\}^{\eta_k}$.
Using the fact that $\eta$ and $\eta_k\setminus\eta$ are independent, it is for $\epsilon'\le1-\sfrac{1}{k}$ and $\epsilon=\epsilon'/(1-\sfrac{1}{k})$ straightforward to verify that $(\eta,\eta(\epsilon'))$ and $(\eta,\eta^\epsilon)$ are equal in distribution.

\par The two-stage construction thus leads us to the identity
\begin{equation}\label{eq:discretization}
\begin{split}
&\EE_{n, \sfrac{1}{2}}\big[ f_{R}(\eta)f_{R}(\eta(\epsilon'))\big]- \EE_{n,\sfrac{1}{2}}\big[f_{R}(\eta)\big]^{2}\, =\,\var_{kn,\sfrac{1}{2}}\big(\EE\left[f_{R}(\eta)|\eta_{k}\right]\big)\\
&\quad\quad\quad\quad+\EE_{kn, \sfrac{1}{2}}\big[\EE[f_{R}(\eta)f_{R}(\eta^{\epsilon})| \eta_{k}]-\EE[f_{R}(\eta)|\eta_{k}]^{2}\big].
\end{split}
\end{equation}
In order to prove that $f_R$ is noise sensitive it will thus suffice to prove that each term in the right-hand side of~\eqref{eq:discretization} is small for large $n$. To prove that the variance term, for fixed $k$, tends to zero as $n$ tends to infinity turns out to be equivalent to the original problem. To see this, let $\eta'$ and $\eta''$ be obtained independently from $\eta_k$ by keeping each point with probability $\sfrac1k$. Then, for $\epsilon'=1-\sfrac1k$ the joint law of $(\eta',\eta'')$ equals that of $(\eta,\eta(\epsilon'))$, and hence\footnote{Here, $k>1$ does not have to be an integer.}
\begin{equation}\label{eq:thinning}
\begin{aligned}
&\EE_{n, \sfrac{1}{2}}\big[ f_{R}(\eta)f_{R}(\eta(\epsilon'))\big]- \EE_{n,\sfrac{1}{2}}\big[f_{R}(\eta)\big]^{2}\\
&\qquad=\,\EE_{kn,\sfrac12}\big[\EE[f_R(\eta')f_R(\eta'')|\eta_k]\big]-\EE_{kn,\sfrac12}\big[\EE[f_R(\eta')|\eta_k]\big]^2\\
&\qquad=\,\var_{kn,\sfrac12}\big(\EE[f_R(\eta')|\eta_k]\big).
\end{aligned}
\end{equation}
However, we shall in Lemma~\ref{lemma:variance_decay} see that the expression in~\eqref{eq:thinning} tends to zero as $k\to\infty$. The goal will then be to show that, for large $k$, conditional on $\eta_k$, the function $f_R:\{0,1\}^{\eta_k}\to\{0,1\}$ is noise sensitive in the sense of~\eqref{eq:ns_def}, with high probability.

\par In a similar manner we shall rely on the two-stage construction in order to prove that $f_R$ has a sharp threshold at $p=\sfrac12$. The construction here will have to be slightly different, since we now want to vary the colour of certain points and not their presence. We will thus let $\overline\eta_k$ denote the projection of $\eta_k$ to $S$, and instead aim to show that $\PP[f_R(\eta)=1|\overline\eta_k]$ grows from 0 to 1 in a narrow interval around $p=\sfrac12$, with high probability.
A first step in both these instances is obtained in the following lemma, which has its origins in~\cite{abgm}, although the proof we present here is taken from~\cite{att}.

\begin{lemma}\label{lemma:variance_decay}
For every integer $k \geq 2$ and $p \in (0,1)$ we have
\begin{equation*}
\var_{kn, p}\big(\EE\left[f_{R}(\eta)|\eta_{k}\right]\big) \leq \frac{1}{k}.
\end{equation*}
\end{lemma}

\begin{proof}
It all boils down to use a suitable construction for the pair $(\eta_{k},\eta)$. Consider $k$ independent copies $\eta^{(1)},\eta^{(2)},\ldots,\eta^{(k)}$ of $\eta$, and let $\kappa$ be chosen uniformly in $[k]$.
We then observe that
\begin{equation*}
\begin{aligned}
\var_{kn, p}\left(\EE\left[f_{R}(\eta)|\eta_{k}\right]\right) \,& \leq\, \var_{n,p}\big(\EE\left[f_{R}(\eta^{(\kappa)})|(\eta^{(i)})_{i=1}^{k}\right]\big)\\
& =\, \var_{n,p}\bigg(\frac{1}{k}\sum_{i=1}^{k}f_{R}(\eta^{(i)})\bigg).
\end{aligned}
\end{equation*}
The lemma then follows from the independence of the $\eta^{(i)}$.
\end{proof}

\par As an easy corollary of the lemma above we obtain the following.

\begin{lemma}\label{lemma:horizontal_crossing}
For every rectangle $R \subseteq S$ there exists $k_{0}$, depending only on the aspect ratio of $R$, such that if $k \geq k_{0}$, then we have, for all large $n$, that
\begin{equation*}
\PP_{kn,\sfrac{1}{2}} \Big[\PP\left[ f_R(\eta)=1| \eta_{k}\right]  \notin [\sfrac{\useconstant{c:non_trivial}}{2}, 1-\sfrac{\useconstant{c:non_trivial}}{2}] \Big] \leq \frac{1}{\sqrt{k}},
\end{equation*}
where $\useconstant{c:non_trivial}$ is the constant in~\eqref{eq:non_trivial}.
\end{lemma}

\begin{proof}
Chebyshev's inequality and Lemma~\ref{lemma:variance_decay} imply that
\begin{equation*}
\begin{split}
\PP_{kn, \sfrac{1}{2}} & \Big[\PP\left[f_R(\eta)=1| \eta_{k}\right]  \notin [\sfrac{\useconstant{c:non_trivial}}{2}, 1-\sfrac{\useconstant{c:non_trivial}}{2}] \Big]\\
& \qquad\leq 
\PP \left[\Big|\PP\left[f_R(\eta)=1| \eta_{k}\right]-\PP_{n, \sfrac{1}{2}}[f_R=1]\Big| \geq \frac{\useconstant{c:non_trivial}}{2}\right] \\
& \qquad \leq \frac{4}{\useconstant{c:non_trivial}^{2} k} \leq \frac{1}{\sqrt{k}},
\end{split}
\end{equation*}
for $k$ and $n$ large enough.
\end{proof}

\begin{remark}\label{remark:crossing_fixed_k}
Notice that if, for some $k \geq 2$, we have
\begin{equation*}
\lim_{n \to \infty}\var_{kn, \sfrac{1}{2}}\big(\EE[f_R(\eta)\,|\,\eta_{k}]\big)=0,
\end{equation*}
then the conclusion in Lemma~\ref{lemma:horizontal_crossing} strengthens to
\begin{equation*}
\lim_{n \to \infty}\PP_{kn, \sfrac{1}{2}}\big[\PP\left[\left. f_R(\eta)=1\right| \eta_{k}\right]  \notin [\sfrac{\useconstant{c:non_trivial}}{2}, 1-\sfrac{\useconstant{c:non_trivial}}{2}] \big] =0.
\end{equation*}
\end{remark}

\section{An algorithm with low revealment}\label{sec:algorithms}

\par In this section, we continue to work towards the proofs of Theorems~\ref{teo:ns} and~\ref{teo:tw}. We will adopt the two-stage construction introduced in the previous section, and devise an algorithm which, conditional on the denser set of points $\eta_k$, determines the outcome of $f_R(\eta)$ by querying points of $\eta_k$ whether they are contained in the sparser set $\eta$. We then proceed to show that this algorithm has low revealment, which in the next section will allow us to deduce that $f_R$ is noise sensitive and has a threshold at $p=\sfrac12$.

\subsection{The algorithm}

\par In this subsection we describe the algorithm. Loosely speaking, it will explore the square $S$ until it has discovered all blue components that touch a randomly selected vertical line through $R$.
This is achieved by querying points close to the vertical line first, and then proceeding to points that are close to already explored blue components connected to the vertical line. Since we cannot tell the Voronoi tessellation of $\eta$ by just observing $\eta_k$, we will only gain information about the actual tiling locally as we go.
To contour this difficulty, we will split $S$ into boxes on a mesoscopic scale (see Figure~\ref{fig:partition}), so that by querying all points within such a box we will correctly determine the tiling within that box with high probability, apart from close to the boundary. That is, by further dividing each box into nine sub-boxes the thus learn the tiling of $\eta$ correctly within the centre box with high probability.

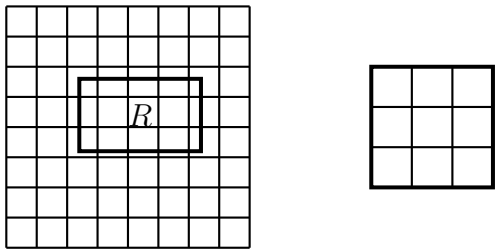
\begin{figure}[!ht]
\centering
\begin{tikzpicture}[scale=0.8]

\foreach \k in {0,0.5,1,1.5,2,2.5,3,3.5,4}{
	\draw[thick] (\k,0)--(\k,4);
	\draw[thick] (0, \k)--(4,\k);
	}
	
\draw[line width=1.5](1.2,1.6)rectangle(3.2,2.8);
\node at (2.2,2.2){$R$};

\draw[line width=1.5](6,1)rectangle(8,3);
\foreach  \k in {1/3,2/3}{
	\draw[thick](6+2*\k,1)--(6+2*\k,3);
	\draw[thick](6, 1+2*\k)--(8,1+2*\k);
	}

\end{tikzpicture}
\caption{The unit square divided into smaller squares at a mesoscopic scale. When all points of $\eta_k$ in a sub-square are queried, then the tiling within the center box in a further division into nine sub-boxes is correctly determined with high probability.}
\label{fig:partition}
\end{figure}

\par If the algorithm discovers a blue component that touches both left and right sides of $R$, then there is a horizontal blue crossing of $R$. If not, then there is a vertical red crossing. The reason the algorithm has low revealment is that a given point is both unlikely to be close to the randomly located vertical line, and unlikely to be connected far by a blue path.

\par The rest of this section will be dedicated to confirming these claims. First we give a more precise description of our algorithm, see Algorithm~\ref{alg:crossing}. Recall that $\Omega$ is the collection of finite subsets of $S \times \{0,1\}$.

\begin{algorithm}[!h]\caption{(Existence of a horizontal blue crossing)}\label{alg:crossing}
\begin{algorithmic}[1]
\State \textbf{Input:} $\eta_{k} \in \Omega$, $\eta \in \{0,1\}^{\eta_{k}}$ and $R=[a,b] \times [c,d] \subseteq S$.
\State Choose a point $x_{0}$ uniformly in the mid third of the interval $[a, b]$.
\State Consider a lattice in $S$ with mesh size $m= 1/\lceil n^{\sfrac{1}{4}} \rceil$, and divide each cell in this lattice into nine equally sized subcells.
\State Query points in all cells of the lattice that intersect $R \cap \{x=x_0\}$ and their neighbouring cells. Declare the examined cells \emph{explored}, and each explored cell \emph{safe} if also the eight cells that surround it are explored.
\State If any of the cells explored so far contains an empty subcell, then query all points of $\eta_{k}$. Otherwise, proceed and explore all cells that share an edge with a safe cell and are connected to the line $\{x=x_0\}$ by a blue component inside the safe region. Explore also any cell neighbouring to these cells and declare an explored cell which is surrounded by explored cells safe.
\State Repeat Step 5 until all connected blue components inside $R$ that intersect $\{x=x_0\}$ are discovered. If there is a connected blue component inside $R$ that connects $\{x=a\}$ to $\{x=b\}$, return 1. Otherwise, return 0.
\end{algorithmic}
\end{algorithm}

\begin{lemma}
Algorithm~\ref{alg:crossing} determines the outcome of $f_{R}$ almost surely.
\end{lemma}

\begin{proof}
Observe that if there exists a horizontal blue crossing of $R$, then it necessarily crosses every vertical line through $R$. Hence, it suffices to verify that given $\eta_k$ the algorithm correctly determines all connected blue components of $\eta$ inside $R$ that intersect the random vertical line $\{x=x_0\}$.

If the algorithm queries all points of $\eta_k$ then this is trivially true. If not, then all we need to verify is that for each safe cell, i.e., a cell which is explored along with its eight surrounding neighbours, we have determined the tiling within. This is indeed the case since if no neighbouring cell has an empty subcell, then no point outside the safe cell and its eight neighbours can affect the tiling inside the safe cell.
\end{proof}

\par Now that we have an algorithm that determines $f_{R}$, we need to bound its revealment. Since the algorithm only reveals the configuration inside cells of a mesoscopic lattice, we consider each such cell individually and bound the revealment of every point inside it at once. This is done in the next two subsections.

\subsection{One-arm estimates}

\par There are three possibilities for a point in $\eta_k$ to be queried by the algorithm above. First of all, there is the case that some subcell of a cell does not contain any point of $\eta$. In this case, all the points of $\eta_{k}$ are queried. The second case is when the point is contained in a cell `close' to the random vertical line through $R$. Finally, there is the possibility that the given points is in a cell located `far' from the line, but there exists a connected blue path in $\eta$ connecting the vertical line with one of the eight cells that surround that cell. In this subsection we shall bound the probability of the third of these possibilities.

\par Let $m=\lceil n^{\sfrac{1}{4}} \rceil^{-1}$ as before, and partition $S$ into squares of side length $m$. The precise choice of $m$ is irrelevant as long as $n^{-\sfrac{1}{2}} \ll m \ll 1$. Let $C\subseteq S$ be a cell in this lattice, and let $C'$ be the square of side length $3m$ centered at $C$. We define $\Arm(C)$ as the event that there exists a blue path that connects $C'$ to the boundary of the square of side $\sqrt{m}$ centered at $C$.

\begin{proposition}\label{prop:one_arm}
There exists $\delta>0$ such that, for every $\gamma>0$, we can find $k_{0}\ge1$ so that, for $k \geq k_{0}$, $p\le\sfrac12$, and all large $n$, depending on $k$, we have
\begin{equation*}
\PP_{kn,p}\Big[\PP\left[ \eta \in \Arm(C)\,|\,{\eta}_{k}\right] > n^{-\delta} \Big]< n^{-\gamma}.
\end{equation*}
\end{proposition}

\par Estimates of this type have previously been obtained in~\cite{abgm,qvp,att}, and the proof presented here will be similar, although different in some details. It will suffice to consider the critical case $p=\sfrac12$ due to monotonicity.
As a first step, we prove a lemma that bounds the probability that a configuration contains a large cell. Let
\newconstant{c:large_cell}
\begin{equation}\label{eq:Hdef}
E:=\big\{\text{some cell of $\eta$ has radius larger than } n^{-\sfrac{1}{3}}\big\}.
\end{equation}

\begin{lemma}\label{lemma:large_cell}
There exists $\useconstant{c:large_cell}>0$ such that, for all $n\ge1$, we have
\begin{equation}\label{eq:large_cell}
\PP_{n, \sfrac{1}{2}}[E]\leq\exp({-\useconstant{c:large_cell}n^{\sfrac13}}).
\end{equation}
\end{lemma}
\begin{proof}
We split the unit square $S$ into boxes of side length $(10\lceil n^{\sfrac{1}{3}} \rceil)^{-1}$. Notice that for $E$ to occur it is necessary for the intersection of $\eta$ with at least one of these about $100n^{2/3}$ boxes to be empty.  For each individual box this occurs with probability at most $\exp(-0.01 n\cdot n^{-\sfrac23})$.
Via the union bound we conclude that
\begin{equation*}
\PP_{n, \sfrac{1}{2}}[E]\,\le\, 100\,n^{2/3}\exp({-0.01 n^{\sfrac13}}),
\end{equation*}
as required.
\end{proof}

\begin{proof}[Proof of Proposition~\ref{prop:one_arm}.]
Fix a cell $C \subseteq S$ of side length $m=\lceil n^{\sfrac{1}{4}} \rceil^{-1}$.
For every integer $j \ge0$, denote by $A_{j}$ the square annulus centered around $C$, with inner side-length $4^{j}m$ and outer side-length $3\cdot 4^{j}m$. Let $O_j$ be the event that there is \emph{not} a blue path connecting the inner and outer boundary of $A_j$. That is, $O_j$ is the even that there is a red path in $A_j$ that disconnects any blue component touching $C$ from the exterior of $A_j$.
Observe that, in order for the event $\Arm(C)$ to occur, $O_j$ cannot occur for integers $j$ in the set
\begin{equation*}
J:=\big\{j \in \NN : m \leq 4^{j}m \leq m^{\sfrac{1}{2}}\big\}.
\end{equation*}

\begin{figure}
\centering
\begin{tikzpicture}[scale=0.7]

\draw[pattern=north west lines] (-3,-3) rectangle (3,3);
\draw[fill=white](-2,-2) rectangle (2,2);
\draw[pattern=north west lines](-1,-1) rectangle (1,1);
\draw[fill=white](-0.5,-0.5)rectangle(0.5,0.5);

\foreach \x in {0.25,0.5,1,2,3,4}
{
\draw[thick,](-\x,-\x)rectangle(\x,\x);
}

\node at (0,0){$C$};
\end{tikzpicture}
\caption{The square $C$, surrounded by a larger square with side length $m^{\sfrac{1}{2}}$. The dashed annuli represent the sets $A_{j}$. Notice that, if there is a blue path from $C$ to the boundary of the square, none of the annuli can contain a red circuit.}
\label{fig:one_arm}
\end{figure}
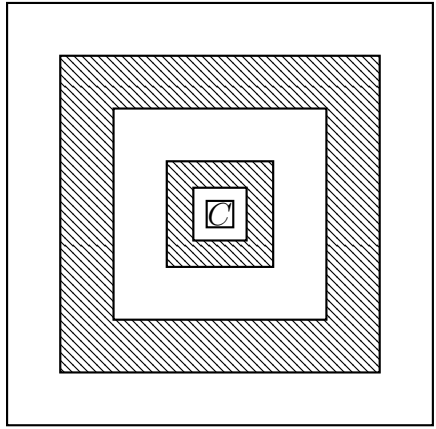

Let $E$ be the event in~\eqref{eq:Hdef}, and let $A_j'$ denote the set of points within distance $m/3$ of $A_j$. We note that, on $E^c$, the events $O_j$ are determined by the restriction of $\eta$ to $A_j'$, which we shall denote $\eta^{(j)}$. That is, if $g_j:\Omega\to\{0,1\}$ denotes the indicator of $O_j$, then
\begin{equation}\label{eq:g_j}
{\bf1}_{E^c}\cdot g_j(\eta)={\bf1}_{E^c}\cdot g_j(\eta^{(j)}).
\end{equation}
Moreover, since the sets $A_j'$ are disjoint the configurations $\eta^{(j)}$ are independent. Since $O_j$ cannot occur for any $j\in J$ in case that $\Arm(C)$ occurs, it follows that
\begin{equation}\label{eq:arm_bound}
\begin{aligned}
\PP[\Arm(C)\,|\,\eta_k]\,&\le\,\PP[E\,|\,\eta_k]+\PP\Big[E^c\cap\bigcap_{j\in J}O_j^c\,\Big|\,\eta_k\Big]\\
&\le\,\PP[E\,|\,\eta_k]+\prod_{j\in J}\PP\big[g_j(\eta^{(j)})=0\,\big|\,\eta_k\big]\\
&\le\,\PP[E\,|\,\eta_k]+\prod_{j\in J}\Big(\PP[O_j^c\,|\,\eta_k]+\PP[E\,|\,\eta_k]\Big).
\end{aligned}
\end{equation}
Recall the constant $\useconstant{c:non_trivial}>0$, from~\eqref{eq:non_trivial}, and introduce the events
\begin{equation*}
D:=\left\{\PP[E\,|\,\eta_{k}] \geq \sfrac1n \right\}\quad\text{and}\quad
D_{j}:=\big\{\PP[O_{j}\,|\,\eta_{k}] \leq \useconstant{c:non_trivial}^{4}/32-2/n\big\},
\end{equation*}
and let $D^*$ denote the event that $D_j$ occurs for at least half the indices in $J$. From~\eqref{eq:arm_bound} we conclude that on $(D^*\cup D)^c$ there exists $\delta>0$ such that
$$
\PP[\Arm(C)|\eta_k]\,\le\,\sfrac1n+\big[(1-{\useconstant{c:non_trivial}^{4}}/{32})+\sfrac3n\big]^{\sfrac{|J|}{2}}\,\le\, n^{-\delta}.
$$

It remains to bound the probability that either $D$ or $D^*$ occurs.
By Markov's inequality and Lemma~\ref{lemma:large_cell},
\begin{equation}\label{eq:bound_event_2}
\PP_{kn, \sfrac{1}{2}}[D] \,\leq\, n\,\PP_{n,\sfrac12}[E]\,\leq\, n\cdot\exp(-\useconstant{c:large_cell}n^{\sfrac{1}{3}}).
\end{equation}
Since $nm^2\gg1$ and the annulus $A_j$ is the union of four rectangles with sides $3\cdot4^jm$ and $4^jm$, it follows from Lemma~\ref{lemma:horizontal_crossing} and Harris' inequality that
\begin{equation}\label{eq:bound_event_1}
\PP_{kn, \sfrac{1}{2}}\big[\PP[O_j|\eta_k]\le\useconstant{c:non_trivial}^{4}/16\big] \leq 4k^{-\sfrac{1}{2}}.
\end{equation}
Here, it is important to observe that the bound is independent of the chosen annulus. Indeed, if the annulus is not entirely contained in $S$, then it would only be harder for blue to reach its outer boundary from within.

We then observe that
\begin{equation}\label{eq:DuD*}
\PP_{kn, \sfrac{1}{2}}[D\cup D^*] \,\leq\, \PP_{kn, \sfrac{1}{2}}[D]+2^{\sfrac{|J|}{2}}\sup_{I}\PP_{kn, \sfrac{1}{2}}\Big[D^{c} \cap \bigcap_{j \in I} D_{j}\Big],
\end{equation}
where the supremum above is taken over all subsets of $J$ with at least $|J|/2$ elements.
Repeated use of~\eqref{eq:g_j} shows that
\begin{equation*}
\begin{split}
\PP_{kn, \sfrac{1}{2}}\Big[D^{c} \cap \bigcap_{j \in I} D_{j}\Big] \,& \leq\, \PP_{kn, \sfrac{1}{2}}\Big[\bigcap_{j \in I}\Big\{\PP\left[g_j(\eta^{(j)})=1\,|\,\eta_{k}\right] \leq \frac{\useconstant{c:non_trivial}^{4}}{16}-\sfrac1n\Big\}\Big] \\
& \leq\, \prod_{j \in I}\PP_{kn, \sfrac{1}{2}}\Big[\PP\left[g_j(\eta^{(j)})=1\,|\,\eta_{k}\right] \leq \frac{\useconstant{c:non_trivial}^{4}}{16}-\sfrac1n\Big] \\
&\leq\,\prod_{j\in I}\Big(\PP_{kn,\sfrac12}[D]+\PP_{kn,\sfrac12}\Big[\PP[O_j|\eta_k]\le\frac{\useconstant{c:non_trivial}^{4}}{16}\Big]\Big)\\
\end{split}
\end{equation*}
Hence, combined with the estimates in~\eqref{eq:bound_event_2}-\eqref{eq:DuD*} we conclude that
\begin{equation*}
\PP_{kn, \sfrac{1}{2}}[D\cup D^*] \,\leq\, n\cdot\exp(-\useconstant{c:large_cell}n^{\sfrac13})+2^{\sfrac{|J|}{2}}\big[n\cdot\exp(-\useconstant{c:large_cell}n^{\sfrac13})+4k^{-\sfrac{1}{2}}\big]^{\sfrac{|J|}{2}}.
\end{equation*}
Since $|J|=\Omega(\log n)$ we may for every $\gamma>0$ choose $k$ large so that the above estimate is bounded by $n^{-\gamma}$ for all large $n$.
\end{proof}

\subsection{Revealment of the algorithm}

\par Now that we have the one-arm estimate, we can bound the revealment of our algorithm. We recall that a point in $\eta_k$ may be queried if the $m\times m$ cell in which it belongs is either `close' to the random vertical line through $R$, or `far' but connected by  a blue path to that line, or if the algorithm at some point discovers a subcell of some $m\times m$ cell which is empty. 

\begin{proposition}\label{prop:revealment}
Let $\mathcal{A}$ denote Algorithm~\ref{alg:crossing}. There exist $\delta>0$ and $k_0\ge1$ such that, for every $k \geq k_{0}$, $p\le\sfrac12$, and all large $n$, we have
\begin{equation*}
{\PP}_{kn,p}\big[\delta_{\sfrac{1}{k}}(\mathcal{A},{\eta}_{k})> n^{-\delta}\big] < n^{-50}.
\end{equation*}
\end{proposition}

\begin{proof}
As before we partition the unit square $S$ into cells of side length $m$, and split each cell $C$ into nine further subcells. Let $G$ be the event that each such subcell contains a point of $\eta$, and let
$$
B:=\big\{\PP[G^c|\eta_k]>\sfrac1n\big\}.
$$
Markov's inequality then gives that, for large $n$,
$$
\PP_{kn,p}[B]\,\le\, n\,\PP_{n,p}[G^c]\,\le\, n\cdot 9m^{-2}\exp(-9^{-1}nm^2)\,<\,\frac12n^{-50}.
$$

Next we fix $\gamma=100$ and let $\delta>0$ and $k_0\ge1$ be as in Proposition~\ref{prop:one_arm}. Let $B'$ denote the event that for some $m\times m$ cell $C$, we have $\PP[\Arm(C)|\eta_k]>n^{-\delta}$. The union bound and Proposition~\ref{prop:one_arm} then gives that for large $n$
$$
\PP_{kn,p}[B']\,\le\, m^{-2}\max_{C\subseteq S}\PP_{kn,p}\big[\PP[\Arm(C)|\eta_k]>n^{-\delta}\big]\,<\, \frac12n^{-50}.
$$

For a given $m\times m$ cell $C$ we let $D_C$ be the event that $C$ is within distance $2\sqrt{m}$ of the random line through $R$. The probability of $D_C$ is independent of $\eta_k$, and one can obtain an upper bound of order $\sqrt{m}$, uniformly in $C$.

For a point of $\eta_k$ to be queried there has either to exist a subcell of some $m\times m$ cell that is empty, or the point must lie in a cell $C$ within distance $2\sqrt{m}$ of the randomly chosen vertical line through $R$, or $\Arm(C)$ has to occur. The revealment of $\mathcal{A}$ thus has to satisfy
$$
\delta_{\sfrac1k}(\mathcal{A},\eta_k)\,\le\,\max_{C\subseteq S}\Big(\PP[G^c|\eta_k]+\PP[D|\eta_k]+\PP[\Arm(C)|\eta_k]\Big),
$$
which restricted to the event $(B\cap B')^c$ is at most $n^{-1}+n^{-\sfrac18}+n^{-\delta}$.
\end{proof}

We may analogously to the algorithm $\mathcal{A}$ define an algorithm $\mathcal{A}'$ which looks for a vertical red crossing of $R$. By symmetry it follows that, for $p\ge\sfrac12$,
$$
\PP_{kn,p}\big[\delta_{\sfrac1k}(\mathcal{A}',\eta_k)>n^{-\delta}\big]<n^{-50}.
$$

\section{Noise sensitivity and the threshold window}\label{sec:nstw}

\par This section is devoted to the proofs of Theorems~\ref{teo:ns} and~\ref{teo:tw}. First, we prove Theorem~\ref{teo:ns}, that Voronoi percolation is noise sensitive, with a positive noise sensitivity exponent. Then we bound the width of the threshold window, proving Theorem~\ref{teo:tw}. Throughout the section we work with the two-stage construction of the random Voronoi configuration, as described in Section~\ref{sec:cont_to_disc}.

\begin{proof} [Proof of Theorem~\ref{teo:ns}]
Due to Equation~\eqref{eq:discretization} and Lemma~\ref{lemma:variance_decay} it will suffice, for the first part of the theorem, to show that for some $\gamma>0$ and all large $k$ we have
\begin{equation}\label{eq:discrete_ns}
\EE_{kn, \sfrac{1}{2}}\Big[\EE[f_{R}(\eta)f_{R}(\eta^{\epsilon_n})| \eta_{k}]-\EE[f_{R}(\eta)|\eta_{k}]^{2}\Big] \to 0\quad\text{as }n\to\infty,
\end{equation}
where $\epsilon_n=n^{-\gamma}$.

Let $\mathcal{A}$ be the algorithm in Algorithm~\ref{alg:crossing}.
The Schramm-Steif revealment theorem (Proposition~\ref{prop:ssrt}) gives that, for almost every $\eta_k$ and $m\ge1$, we have
$$
\EE[f_{R}(\eta)f_{R}(\eta^{\epsilon_n})| \eta_{k}]-\EE[f_{R}(\eta)|\eta_{k}]^{2}\,\le\,\exp(-\epsilon_nm)+m^2\delta_{\sfrac12}(\mathcal{A},\eta_k).
$$
Let $\delta>0$ be as in Proposition~\ref{prop:revealment}, and let $B_n$ denote the event that $\delta_{\sfrac{1}{k}}(\mathcal{A},\eta_{k})>n^{-\delta}$. Then $\PP_{kn,\sfrac12}[B_n]<n^{-50}$, and consequently
\begin{equation*}
\begin{split}
\EE_{kn, \sfrac{1}{2}} & \Big[\EE[f_{R}(\eta)f_{R}(\eta^{\epsilon_n})| \eta_{k}]-\EE[f_{R}(\eta)|\eta_{k}]^{2}\Big] \\
& \qquad \leq n^{-50}+\exp(-\epsilon_n m)+m^{2}\EE_{kn, \sfrac{1}{2}}[\delta_{\sfrac{1}{k}}(\mathcal{A},\eta_{k})\charf{B_{n}^{c}}]\\
& \qquad \leq n^{-50}+\exp(-\epsilon_n m)+m^{2}n^{-\delta}.
\end{split}
\end{equation*}
Hence,~\eqref{eq:discrete_ns} holds with $\gamma=\delta/3$ and $n^{\delta/3}\ll m\ll n^{\delta/2}$, which concludes the proof of Theorem~\ref{teo:ns}.
\end{proof}

We proceed with the proof of Theorem~\ref{teo:tw}.

\begin{proof}[Proof of Theorem~\ref{teo:tw}]
Given $\eta_k\in\Omega$ we shall with $\overline\eta_k$ denote its projection onto $S$. We first note that by dominated convergence we have
\begin{equation}\label{eq:change_order}
\frac{d}{dp}\PP_{n,p}[f_R=1]\,=\,\EE_{n,p}\Big[\frac{d}{dp}\PP[f_R(\eta)=1|\overline\eta_k]\Big],
\end{equation}
since the rate at which $\PP[f_R(\eta)=1|\overline\eta_k]$ may increase as $p$ varies is bounded by the number of variables $|\overline\eta_k|$ affected by $p$.
Moreover, given $\overline\eta_k$, we may think of $\eta$ as an element in $\{0,1\}^{\overline\eta_k}\times\{0,1\}^{\overline\eta_k}$, where the first half of the coordinates determine `colour' and the second half determine `presence' in the final configuration. The Margulis-Russo formula then gives that
\begin{equation*}
\frac{d}{dp}\PP[f_R(\eta)=1\,|\,\overline{\eta}_{k}] \, =\, \sum_{x \in \overline{\eta}_{k}}\PP\left[\left. \begin{array}{c}
x \text{ is present and its colour} \\ \text{is pivotal for } f_{R} \end{array}\,\right|\, \overline{\eta}_{k}\right]
\end{equation*}
almost surely. Since a blue point is better than no point, and no point is better than a red point, it follows that switching presence rather than colour of a point is less likely to affect the outcome of $f_R$. Consequently, the derivative is bounded from below by the sum
$$
\sum_{x \in \overline{\eta}_{k}}\PP\big[x \text{ is present and its presence is pivotal for } f_{R}\,\big|\, \overline{\eta}_{k}\big].
$$
Each term in the above expression can be rewritten as $\frac1k\EE[\Inf_x^{\sfrac1k}(f_R,\eta_k)|\overline\eta_k]$, where the factor $\sfrac1k$ comes from the probability of being present. Hence,~\eqref{eq:change_order} and the OSSS inequality (Proposition~\ref{prop:osss}) together give that
\begin{equation}\label{eq:deriv_bound}
\frac{d}{dp}\PP[f_R(\eta)=1] \, \ge\, \frac1k\EE\Big[\sum_{x \in {\eta}_{k}}\Inf_x^{\sfrac1k}(f_R,\eta_k)\Big]\,\ge\,\frac4k\,\EE\bigg[\frac{\var(f_R|\eta_k)}{\delta_{1/k}(\mathcal{A},\eta_k)}\bigg].
\end{equation}

Fix $\epsilon>0$ and let $I_\epsilon=I_\epsilon(n)$ denote the set of points $p\in[0,1]$ for which $\PP_{n,p}[f_R=1]\in(\epsilon,1-\epsilon)$. By monotonicity $I_\epsilon$ is an interval, and for small $\epsilon$ the interval contains the point $\sfrac12$. Consequently, to complete the proof it will suffice to show that there exists $\gamma>0$ such that $|I_\epsilon|\le n^{-\gamma}$ for all $\epsilon>0$.

Let $\mathcal{A}$ be the algorithm in Algorithm~\ref{alg:crossing}, and $\mathcal{A}'$ be the analogously defined algorithm that looks for a vertical red crossing of $R$. We introduce the events
\begin{equation*}
\begin{aligned}
A&:=\big\{\PP[f_R(\eta)=1|\eta_k]\in(\sfrac\epsilon2,1-\sfrac\epsilon2)\big\},\\
B&:=\big\{\min\{\delta_{\sfrac1k}(\mathcal{A},\eta_k),\delta_{\sfrac1k}(\mathcal{A}',\eta_k)\}<n^{-\delta}\big\},
\end{aligned}
\end{equation*}
with which~\eqref{eq:deriv_bound} reduces to
\begin{equation}\label{eq:deriv_bound_2}
\frac{d}{dp}\PP_{n,p}[f_R=1]\,>\,\epsilon^2k^{-1}n^\delta\,\PP_{n,p}[A\cap B].
\end{equation}

Next we fix $k\ge16/\epsilon^2$. By Chebyshev's inequality and Lemma~\ref{lemma:variance_decay} we then have, for all $p\in I_\epsilon$, that
$$
\PP_{n,p}[A^c]\,\le\,(2/\epsilon)^2\var_{n,p}\big(\PP[f_R(\eta)=1|\eta_k]\big)\,\le\,4/(\epsilon^2k)\,\le\,\sfrac14.
$$
By increasing $k$ if necessary, Proposition~\ref{prop:revealment} gives that $\PP_{n,p}[B^c]\le\sfrac14$ for all $p\in[0,1]$ and $n$ large. Integrating over $I_\epsilon$ in~\eqref{eq:deriv_bound_2} thus leads to the bound
$$
1\,\ge\,\int_{I_\epsilon}\frac{d}{dp}\PP_{n,p}[f_R=1]\,dp\,\ge\,\frac12\epsilon^2k^{-1}n^{\delta}|I_\epsilon|,
$$
and hence that $|I_\epsilon|\le2k/(\epsilon^2n^\delta)$. Since $\epsilon>0$ was arbitrary, the theorem follows with $\gamma=\delta/2$.
\end{proof}

\par We can also study the behaviour of $f_{R}(\eta)$ for fixed values of $k$.

\begin{proposition}\label{prop:fixed_k}
For every $p\in[0,1]$ and $k>1$, not necessarily an integer,
\begin{equation*}
\lim_{n\to\infty}\var_{kn,p}\big(\EE\left[f_{R}(\eta)|\eta_{k}\right]\big)=0.
\end{equation*}
Besides, there exists $\delta>0$ such that for all $p\in[0,1]$, $k>1$ and all large $n$
\begin{equation*}
\PP_{kn,p}\bigg[ \delta_{\sfrac1k}(\mathcal{A},\eta_k) > n^{-\delta}\bigg] < n^{-50}.
\end{equation*}
\end{proposition}

\begin{proof}
For $p=\sfrac12$ the first statement of the proposition is immediate from~\eqref{eq:thinning} and Theorem~\ref{teo:ns}. For $p\neq\sfrac12$ it is a trivial consequence of Theorem~\ref{teo:tw}.

As for the second statement, it is necessary to go through the arguments in Section~\ref{sec:algorithms} again, and notice that the only place where $k$ needs to be large is in~\eqref{eq:bound_event_1}. Due to the first part of this proposition, we may modify Lemma~\ref{lemma:horizontal_crossing}, as pointed out in Remark~\ref{remark:crossing_fixed_k}, to obtain that the probability in~\eqref{eq:bound_event_1} is small for every $k>1$ and $n$ large.
\end{proof}

\section{Square-root stability}\label{sec:srt}

\par In Section~\ref{sec:nstw}, we concluded the proof of Theorem~\ref{teo:ns}, and the remainder of this paper will aim to establish Theorem~\ref{teo:other_noises}. The first step in this direction is to establish a result that roughly states that $f_R$ is stable with respect to perturbations that act independently and uniformly on each of the two colours and change at most order square-root of the points.

\par Throughout this section we shall use the notation $\xi:=\{x\in S:(x,0)\in\eta\}$ and $\zeta:=\{x\in S:(x,1)\in\eta\}$ to denote the set of red and blue points respectively, and identify $\eta$ with the pair $(\xi,\zeta)$ when appropriate.

\begin{proposition}\label{prop:srs}
Let $\eta'=(\xi',\zeta')$ and $\eta=(\xi,\zeta)$ be a pair of configurations in $\Omega$, chosen according to $\PP_{n,\sfrac12}$, and whose joint law satisfies the following properties, stated only for the $\xi$-coordinates:
\begin{enumerate}[\quad (i)]\itemsep1pt \parskip0pt \parsep0pt
\item Given $\xi$, the distribution of $\xi \cap \xi'$ is invariant by permutations of $\xi$, and, conditioned on its size, the set $\xi' \setminus \xi$ is formed by independently and uniformly distributed points in $S$.
\item For every $\delta>0$, there exists a constant $C$ such that, for all large $n$,
\begin{equation}\label{eq:assumption_3}
\PP_{n,\sfrac12}\big[|\xi' \bigtriangleup \xi|> C\sqrt{n}\big] < \delta,
\end{equation}
where $\xi' \bigtriangleup \xi$ is the symmetric difference between the two sets.
\end{enumerate}
If, in addition, the pairs $(\zeta, \zeta')$ and $(\xi,\xi')$ are independent, then, for any rectangle $R \subseteq S$, we have
\begin{equation*}
\PP_{n,\sfrac{1}{2}}\big[f_{R}(\zeta,\xi)\neq f_{R}(\zeta',\xi')\big] \rightarrow 0\quad\text{as }n\to\infty.
\end{equation*}
\end{proposition}

\par The square-root scale that figures in the theorem is meaningful in the sense that $\sqrt{n}$ is an upper bound on the derivative of a monotone Boolean function on $n$ bits. Consequently, the threshold window cannot have a width smaller than $1/\sqrt{n}$, and noise sensitive monotone functions have a window that is strictly wider (cf.~\eqref{eq:influence_upper_bound}). Hence, a uniform perturbation that involve order $\sqrt{n}$ bits is therefore too small to affect the outcome of the function.

\par The above heuristic has been made precise in the setting of Boolean functions in a paper by Broman, Garban and Steif~\cite[Lemma~6.1]{exclusion_sensitivity}. We shall prove Proposition~\ref{prop:srs} via a suitable two-stage construction in which a version of the result from~\cite{exclusion_sensitivity} can be applied.

\begin{lemma}\label{lemma:discrete_srs_2}
Let $A_1,A_2,\ldots,A_k$ be a partition of $[n]$, and let $(\omega, \omega^{*})$ be a pair of configurations in $\{0,1\}^{n}$ with law $\overline{\PP}$ satisfying the following properties:
\begin{enumerate}[\quad (i)]\itemsep1pt \parskip0pt \parsep0pt
\item there exists $c>0$ such that $|A_{i}| \geq cn$, for all $i=1,2,\ldots,k$;
\item $\omega$ and $\omega^{*}$ are under $\overline\PP$ uniformly distributed in $\{0,1\}^{n}$;
\item $\overline\PP$ is invariant under all permutations $\pi$ of $[n]$ such that $\pi(A_{i})=A_{i}$, for all $i=1,2,\ldots, k$;
\item for every $\delta>0$ there exists a constant $C$ such that, for all large $n$ and all $i=1,2,\ldots,k$;
\begin{equation*}
\overline\PP\big[\dist_{A_{i}}(\omega,\omega^*)> C \sqrt{|A_{i}|}\big]<\delta,
\end{equation*}
where $\dist_{A_{i}}(\omega,\omega^*):=\sum_{j\in A_{i}}|\omega(j)-\omega^*(j)|$.
\end{enumerate}
Then, for every $\epsilon>0$, there exists a constant $\tilde{C}$ such that, for all large $n$ and any function $f:\{0,1\}^n\to\{0,1\}$, we have
\begin{equation}\label{eq:bgs}
\overline\PP\big[f(\omega) \neq f(\omega^{*})\big] < \epsilon+\frac{\tilde{C}}{\sqrt{n}}\sum_{k\in[n]}\Inf_{k}^{\sfrac{1}{2}}(f).
\end{equation}
\end{lemma}

\par Combined with~\eqref{eq:influence_upper_bound} the bound in~\eqref{eq:bgs} may be expressed in terms of the sum of influences squared or the revealment of algorithms.

\begin{proof}
The case $k=1$ is the statement of Lemma~6.1 in~\cite{exclusion_sensitivity} (with the additional hypothesis that $\omega^*$ is uniform in $\{0,1\}^n$). The remaining cases follows from induction on $k$.

Fix some $k \geq 2$, assume the result is true for all $j \leq k$ and fix a partition $A_{1},A_{2}, \ldots, A_{k+1}$. Denote by $\tilde{A}=[n] \setminus A_{k+1}$ and $\omega=(\omega_{\tilde{A}}, \omega_{A_{k+1}})$ for the restrictions of $\omega$ to the sets $\tilde{A}$ and $A_{k+1}$. Observe that
\begin{equation}\label{eq:bound_lemma_srs_2}
\begin{split}
\overline{\PP}\big[f(\omega_{\tilde{A}},\omega_{A_{k+1}})\neq f_R(\omega_{\tilde{A}}^{*},\omega_{A_{k+1}}^{*})\big] & \leq \overline{\PP}\big[f(\omega_{\tilde{A}},\omega_{A_{k+1}})\neq f_R(\omega_{\tilde{A}}^{*},\omega_{A_{k+1}})\big]  \\
& \,+\overline{\PP}\big[f(\omega_{\tilde{A}}^{*},\omega_{A_{k+1}})\neq f_R(\omega_{\tilde{A}}^{*},\omega_{A_{k+1}}^{*})\big].
\end{split}
\end{equation}

To bound the first probability in the last expression above, we apply the induction hypothesis conditioned on $\omega_{A_{k+1}}$ and use that $\omega_{A_{k+1}}$ is uniformly distributed in $\{0,1\}^{A_{k+1}}$ to obtain
\begin{equation}
\overline{\PP}\big[f(\omega_{\tilde{A}},\omega_{A_{k+1}})\neq f_R(\omega_{\tilde{A}}^{*},\omega_{A_{k+1}})\big] \leq \epsilon+\frac{\tilde{C}}{\sqrt{|\tilde{A}|}}\sum_{k\in \tilde{A}}\Inf_{k}^{\sfrac{1}{2}}(f).
\end{equation}
Analogous computations for the last term in~\eqref{eq:bound_lemma_srs_2} concludes the proof.
\end{proof}

\par We now focus on the proof of Proposition~\ref{prop:srs}.

\begin{proof}[Proof of Proposition~\ref{prop:srs}]
The first step of the proof is to find a suitable construction of the pairs $(\zeta,\zeta')$ and $(\xi, \xi')$. Since the perturbation acts independently on the two colours, this construction can be done separately.

\par For this purpose, let $M=|\xi' \cap \xi|$ and $N=|\xi' \setminus \xi|$. Let $\xi_{2}$ be a Poisson point process on $S$ with intensity measure $n\lambda_S$, and let $\xi$ and $\bar{\xi}$ be uniformly chosen subsets of $\xi_{2}$. Given $|\xi|$, sample the pair $(M,N)$ according to the right conditional law. Next, choose uniformly a subset $\xi^{A} \subseteq \xi$ of size $M$ and let $\xi^{B}$ be a uniformly chosen subset of $\xi_{2} \setminus \xi$ of size $\min\{N, |\xi_{2} \setminus \xi|\}$. Besides, let $\xi^{C}$ be a collection of $N$ independent and uniformly chosen points of the square $S$. Now set
\begin{equation*}
\xi'':= \begin{cases}
\xi^{A} \cup \xi^{B}, & \,\, \text{if } N \leq |\xi_{2} \setminus \xi|, \\
\bar{\xi}, & \,\, \text{if } N > |\xi_{2} \setminus \xi|,
\end{cases}
\end{equation*}
and
\begin{equation*}
\xi''':= \begin{cases}
\xi^{A} \cup \xi^{B}, & \,\, \text{if } N \leq |\xi_{2} \setminus \xi|, \\
\xi^{A} \cup \xi^{C}, & \,\, \text{if } N > |\xi_{2} \setminus \xi|.
\end{cases}
\end{equation*}
Construct the collection $(\zeta, \zeta'', \zeta''')$ analogously, and note that $(\zeta,\zeta''')$ and $(\xi, \xi''')$ have the correct joint distribution.

\par In the next step, we note that $\xi'''=\xi''$ with probability tending to 1. To see this, fix $\epsilon>0$ and notice that $N$ and $|\xi_2\setminus\xi|$ are independent, and that the latter is Poisson with parameter $\sfrac n2$. Then, by assumption~\emph{(ii)} we have
$$
\PP\big[N > |\xi_{2}\setminus \xi|\big]\, \leq\, \PP[N>\sfrac n4] + \PP\big[|\xi_2\setminus\xi| \leq \sfrac n4\big]\, \leq\, \epsilon
$$
for all large $n$. The above construction thus gives, for large $n$, that
\begin{equation}\label{eq:bound_bad_event}
\PP_{n,p}\big[f_{R}(\xi,\zeta)\neq f_{R}(\xi',\zeta')\big]\,\leq\, 2\epsilon+\PP\big[f_R(\xi,\zeta)\neq f_R(\xi'',\zeta'')\big].
\end{equation}

\par Conditional on $(\zeta_2,\xi_2)$ the pairs $(\zeta,\zeta'')$ and $(\xi,\xi'')$ can be thought of as pairs of elements in $\{0,1\}^{\zeta_2}$ and $\{0,1\}^{\xi_2}$ respectively. The last step of the proof will thus be to apply Lemma~\ref{lemma:discrete_srs_2} to bound the last probability above. In preparation for this, set $\delta_m:=\epsilon 2^{-2m}$ and let $C_m$ be the constant in hypothesis~\emph{(ii)} that corresponds to $\delta_m$. Let
$$
B_1:=\big\{\PP\big[|\xi \bigtriangleup \xi''| > C_{m} \sqrt{n}\,\big|\,\xi_{2}\big]\ge 2^{-m}\text{ for some }m\ge1\big\}.
$$
Clearly $|\xi\bigtriangleup\xi''|$ is equal to $|\xi\bigtriangleup\xi'''|$ on the event where $N \leq |\xi_{2}\setminus \xi|$. Hence, the union bound and Markov's inequality give, for large $n$, that
\begin{equation}\label{eq:bound_lemma_hypothesis}
\begin{split}
\PP[B_1]& \, \leq \, \PP[B_{1}, N > |\xi_{2}\setminus \xi|] + \PP[B_{1}, N \leq |\xi_{2}\setminus \xi|] \\
& \, \leq \, \PP[N > |\xi_{2}\setminus \xi|] + \PP\left[ \begin{array}{cl} |\xi \bigtriangleup \xi'''| > C_{m} \sqrt{n}\,\big|\,\xi_{2}\big]\ge 2^{-m} \\\text{ for some }m\ge1 \text{ and } N \leq |\xi_{2}\setminus \xi| \end{array} \right ] \\
&\,\le\,\epsilon+\sum_{m\ge1}2^m\,\PP\big[|\xi \bigtriangleup \xi'''|> C_{m}\sqrt{n}\big] \,\leq\,\epsilon+\sum_{m\ge1}\epsilon2^{-m}\,\le\,2\epsilon.
\end{split}
\end{equation}
Let also $B_2:=\{|\xi_2| \notin [\sfrac n2, 2n]\}$ and define the analogous events $\tilde{B}_{1}$ and $\tilde{B}_{2}$ to the collection $\zeta_{2}$. On the event $G:=(B_{1} \cup B_{2} \cup \tilde{B}_{1} \cup \tilde{B}_{2})^{c}$, Lemma~\ref{lemma:discrete_srs_2} combined with~\eqref{eq:bgs} can be applied and it gives that, for large $n$,
\begin{equation*}
\begin{aligned}
\PP\big[f_R(\xi,\zeta)\neq f_R(\xi'',\zeta'')\big]\,&\le\, 6\epsilon+\EE\big[\PP\big[f_R(\xi,\zeta)\neq f_R(\xi'',\zeta'')\big|\zeta_2,\xi_2\big]{\bf 1}_G\big]\\
&\le\, 6\epsilon+C\,\EE\bigg[\frac{1}{\sqrt{|\eta_2|}}\sum_{x\in\eta_2}\Inf_x^{\sfrac12}(f_R,\eta_2)\bigg].
\end{aligned}
\end{equation*}

By combining the last equation above with~\eqref{eq:bound_bad_event} and~\eqref{eq:influence_upper_bound} we obtain
\begin{equation*}
\PP_{n,\sfrac{1}{2}}\big[f_{R}(\xi,\zeta)\neq f_{R}(\xi',\zeta')\big]\,\leq\,8\epsilon+C\,\EE_{n,\sfrac{1}{2}}\left[\sqrt{\delta_{\sfrac1k}(\mathcal{A},\eta_2)}\right],
\end{equation*}
which by Proposition~\ref{prop:fixed_k} is no larger than $9\epsilon$ when $n$ is large.

Since $\epsilon>0$ was arbitrary, the proof is complete.
\end{proof}

\section{Conservative dynamics and related topics}\label{sec:fp}

\par This final section is devoted to different perturbations in our model.

\bigskip

\par {\bf Thinning and sprinkling.} We begin with a comment on nonconservative and time dependent dynamics. We saw in Section~\ref{sec:nstw} that sensitivity with respect to thinning a configuration uniformly is equivalent to the usual concept of noise sensitivity. We here complement that observation by showing that the same is true for sprinkling.

Let $\eta\in\Omega$ be chosen according to $\PP_{(1-\epsilon)n,\sfrac12}$, and let $\eta'$ and $\eta''$ be independent configurations chosen according to $\PP_{\epsilon n,\sfrac12}$. Then the joint law of $(\eta\cup\eta',\eta\cup\eta'')$ equals that of $(\eta,\eta(\epsilon))$, and
\begin{equation*}
\begin{aligned}
&\EE_{n,\sfrac12}\big[f_R(\eta)f_R(\eta(\epsilon))\big]-\EE_{n,\sfrac12}\big[f_R(\eta)\big]^2\\
&\quad\quad=\,\EE\Big[\EE\big[f_R(\eta\cup\eta')f_R(\eta\cup\eta'')\big|\eta\big]\Big]-\EE\Big[\EE\big[f_R(\eta\cup\eta')\big|\eta\big]\Big]^2\\
&\qquad=\,\var\Big(\EE\big[f_R(\eta\cup\eta')\big|\eta\big]\Big).
\end{aligned}
\end{equation*}
Hence, being sensitive with respect to an $\epsilon$-sprinkling is equivalent to being noise sensitive, and thus follows from Theorem~\ref{teo:ns}. That the same holds for an $\epsilon$-thinning was seen already in Section~\ref{sec:nstw}.

\bigskip

\par {\bf Perturbing the colours.} We shall briefly describe the results in~\cite{qvp}, and explain how they imply that the crossing function is sensitive with respect to re-randomizing a small proportion of the colours of the points. That is, if $\eta'$ is obtained from $\eta\in\Omega$ by resampling the second coordinate of each point $(x,u)\in\eta$ independently and uniformly with probability $\epsilon>0$, then
\begin{equation}\label{eq:qvp}
\EE_{n,\sfrac12}\big[f_R(\eta)f_R(\eta')\big]-\EE_{n,\sfrac12}\big[f_R(\eta)\big]^2\to0\quad\text{as }n\to\infty.
\end{equation}

Given $\eta\in\Omega$, let $\overline\eta$ denote the projection onto $S$. Then,
\begin{equation*}
\begin{aligned}
\EE_{n,\sfrac12}\big[f_R(\eta)f_R(\eta')\big]-\EE_{n,\sfrac12}\big[f_R(\eta)\big]^2&=\EE_{n,\sfrac12}\Big[\EE\big[f_R(\eta)f_R(\eta')\big|\overline\eta\big]-\EE\big[f_R(\eta)\big|\overline\eta\big]\Big]\\
&\quad+\var_{n,\sfrac12}\Big(\EE\big[f_R(\eta)\big|\overline\eta\big]\Big).
\end{aligned}
\end{equation*}
In~\cite{qvp}, the authors show that both expressions in the above right-hand side vanish as $n\to\infty$, and hence prove~\eqref{eq:qvp}. That the variance term tends to zero shows that observing the tiling but not the colouring of a Voronoi configuration typically gives very little information about whether a colouring will typically produce a horizontal blue crossing or not, and confirms a conjecture of Benjamini, Kalai and Schramm~\cite{bks}. The latter is essentially a statement of noise sensitivity of the crossing function in a quenched sense. One may show that noise sensitivity in the sense of~\eqref{eq:ns_voronoi} follows from that statement. However, we emphasize that the techniques used there are more restrictive than the techniques used here, as they are based on a colour-switching trick. It is therefore motivated to present an alternative proof, as we have done here, that applies in a wide range of settings.

\bigskip

\par {\bf Perturbing the positions.} We now turn to the proof of Theorem~\ref{teo:other_noises}. The proof will be based on Proposition~\ref{prop:srs}, which emphasizes a close relation to the exclusion sensitivity studied in~\cite{exclusion_sensitivity}.

\begin{proof}[Proof of Theorem~\ref{teo:other_noises}]
We shall show that the crossing function $f_R$ is sensitive with respect to re-randomizing the positions of a small proportion of the points. This type of perturbation is conservative in the sense that the number of points of each colour is kept constant. Our goal will be to construct the process in a suitable manner, and then apply Proposition~\ref{prop:srs}.

\par As before we shall identify a configuration $\eta\in\Omega$ with a pair of configurations $(\xi, \zeta)$. Let $(X_{i})_{i \ge1}$ and $(Y_{i})_{i \ge1}$ be independent collections of independent and uniformly distributed points in $S$. In addition, let $L$, $M$ and $N$ be independent Poisson distributed random variables with parameters $(1-\epsilon)n/2$, $\epsilon n/2$ and $\epsilon n/2$, respectively. Next we define a triple $(\xi',\xi'',\xi''')$ as
\begin{equation}
\begin{split}
\xi' & := \{X_{1}, X_{2}, \dots, X_{L+M}\},\\
\xi'' & := \{X_{1}, X_{2}, \dots, X_{L}\}\cup \{Y_{1}, Y_{2}, \dots, Y_{N}\},\\
\xi''' & := \{X_{1}, X_{2}, \dots, X_{L}\} \cup \{Y_{1}, Y_{2}, \dots, Y_{M}\}.
\end{split}
\end{equation}
These will be the collection of red points. Finally, we let $(\zeta',\zeta'',\zeta''')$ be an independent copy of $(\xi',\xi'',\xi''')$, that represents the blue points. We consider the three coloured tessellations $\eta'=(\xi', \zeta')$, $\eta''=(\xi'', \zeta'')$ and $\eta'''=(\xi''', \zeta''')$.

\par Notice that the pair $(\eta',\eta'')$ is distributed as the pair $(\eta,\eta(\epsilon))$ in~\eqref{eq:ns_voronoi}, while the pair $(\eta',\eta''')$ is distributed as the pair $(\eta,\eta^*)$ in Theorem~\ref{teo:other_noises}. We also notice that the pair $(\eta'',\eta''')$ satisfy the conditions of Proposition~\ref{prop:srs}. In particular, Chebyshev's inequality shows that for every $\delta>0$ there exists $C$ such that
\begin{equation*}
\PP\big[|\xi'' \bigtriangleup \xi'''|> C\sqrt{n}\big] \,=\,\PP\big[|M-N| > C \sqrt{n}\big] \,\leq\, \frac{\var(M-N)}{C^{2}n} \,\leq\, \frac{\epsilon}{C^{2}} \,\leq\, \delta.
\end{equation*}
Consequently, Proposition~\ref{prop:srs} implies that
\begin{equation}\label{eq:position_perturbation_stability}
\PP_{n,\sfrac12}\big[f_{R}(\eta(\epsilon))\neq f_{R}(\eta^*)\big] \,=\,\PP\big[f_R(\eta'')\neq f_R(\eta''')\big]\,\rightarrow\, 0.
\end{equation}
Finally, we obtain that
\begin{equation*}
\begin{split}
& \left| \EE_{n,\sfrac12}\left[f_{R}(\eta)f_{R}(\eta^*)\right]-\EE_{n,\sfrac12}\left[f_{R}(\eta)\right]^{2} \right|\,\le\, \PP_{n,\sfrac12}\big[f_{R}(\eta(\epsilon))\neq f_{R}(\eta^*)\big]\\
& \qquad\qquad + \left|\EE_{n,\sfrac12}\left[f_{R}(\eta)f_{R}(\eta(\epsilon))\right]-\EE_{n,\sfrac12}\left[f_{R}(\eta)\right]^{2}\right|,
\end{split}
\end{equation*}
which by Theorem~\ref{teo:ns} and~\eqref{eq:position_perturbation_stability} tends to zero as $n\to\infty$.
\end{proof}


\begin{thebibliography}{99}

\bibitem{abgm}
Daniel Ahlberg, Erik Broman, Simon Griffiths, and Robert Morris.
\newblock Noise sensitivity in continuum percolation.
\newblock {\em Israel Journal of Mathematics}, 201(2):847--899, 2014.

\bibitem{qvp}
Daniel Ahlberg, Simon Griffiths, Robert Morris, and Vincent Tassion.
\newblock Quenched {V}oronoi percolation.
\newblock {\em Advances in Mathematics}, 286:889--911, 2016.

\bibitem{ahlste}
Daniel Ahlberg and Jeffrey~E. Steif.
\newblock Scaling limits for the threshold window: {W}hen does a monotone
  {B}oolean function flip its outcome?
\newblock \emph{Ann. Inst. Henri Poincar\'e Probab. Stat.}, to appear. With an
  appendix by G\'abor Pete.

\bibitem{att}
Daniel Ahlberg, Vincent Tassion, and Augusto Teixeira.
\newblock Sharpness of the phase transition for continuum percolation in
  $\mathbb{R}^2$.
\newblock {\em arXiv preprint arXiv:1605.05926}, 2016.

\bibitem{bks}
Itai Benjamini, Gil Kalai, and Oded Schramm.
\newblock Noise sensitivity of {B}oolean functions and applications to
  percolation.
\newblock {\em Publications Math{\'e}matiques de l'Institut des Hautes
  {\'E}tudes Scientifiques}, 90(1):5--43, 1999.

\bibitem{bensch98}
Itai Benjamini and Oded Schramm.
\newblock Conformal invariance of {V}oronoi percolation.
\newblock {\em Comm. Math. Phys.}, 197(1):75--107, 1998.

\bibitem{br}
B{\'e}la Bollob{\'a}s and Oliver Riordan.
\newblock The critical probability for random {V}oronoi percolation in the
  plane is 1/2.
\newblock {\em Probability theory and related fields}, 136(3):417--468, 2006.

\bibitem{exclusion_sensitivity}
Erik~I Broman, Christophe Garban, and Jeffrey~E Steif.
\newblock Exclusion sensitivity of {B}oolean functions.
\newblock {\em Probability theory and related fields}, 155(3-4):621--663, 2013.

\bibitem{dumcopraotas2}
Hugo Duminil-Copin, Aran Raoufi, and Vincent Tassion.
\newblock Exponential decay of connection probabilities for subcritical
  {V}oronoi percolation in $\mathbb{R}^d$.
\newblock Preprint, see \emph{arXiv:\allowbreak 1705.07978}.

\bibitem{dumcopraotas1}
Hugo Duminil-Copin, Aran Raoufi, and Vincent Tassion.
\newblock Sharp phase transition for the random-cluster and {P}otts models via
  decision trees.
\newblock Preprint, see \emph{arXiv:\allowbreak 1705.03104}.

\bibitem{er}
Paul Erd{\H o}s and Alfr{\'e}d R{\'e}nyi.
\newblock On the evolution of random graphs.
\newblock {\em Publ. Math. Inst. Hung. Acad. Sci}, 5(1):17--60, 1960.

\bibitem{fk}
Ehud Friedgut and Gil Kalai.
\newblock Every monotone graph property has a sharp threshold.
\newblock {\em Proceedings of the American Mathematical Society},
  124(10):2993--3002, 1996.

\bibitem{gps}
Christophe Garban, G{\'a}bor Pete, and Oded Schramm.
\newblock The {F}ourier spectrum of critical percolation.
\newblock {\em Acta Mathematica}, 205(1):19--104, 2010.

\bibitem{book_ns}
Christophe Garban and Jeffrey~E Steif.
\newblock {\em Noise sensitivity of {B}oolean functions and percolation},
  volume~5.
\newblock Cambridge University Press, 2014.

\bibitem{kkl}
Jeff Kahn, Gil Kalai, and Nathan Linial.
\newblock The influence of variables on {B}oolean functions.
\newblock In {\em Foundations of Computer Science, 1988., 29th Annual Symposium
  on}, pages 68--80. IEEE, 1988.

\bibitem{kesten80}
Harry Kesten.
\newblock The critical probability of bond percolation on the square lattice
  equals {${1\over 2}$}.
\newblock {\em Comm. Math. Phys.}, 74(1):41--59, 1980.

\bibitem{odonnell}
Ryan O'Donnell.
\newblock {\em Analysis of {B}oolean functions}.
\newblock Cambridge University Press, 2014.

\bibitem{osss}
Ryan O'Donnell, Michael Saks, Oded Schramm, and Rocco~A Servedio.
\newblock Every decision tree has an influential variable.
\newblock In {\em Foundations of Computer Science, 2005. FOCS 2005. 46th Annual
  IEEE Symposium on}, pages 31--39. IEEE, 2005.

\bibitem{odoser07}
Ryan O'Donnell and Rocco~A Servedio.
\newblock Learning monotone decision trees in polynomial time.
\newblock {\em SIAM Journal on Computing}, 37(3):827--844, 2007.

\bibitem{russo82}
Lucio Russo.
\newblock An approximate zero-one law.
\newblock {\em Z. Wahrsch. Verw. Gebiete}, 61(1):129--139, 1982.

\bibitem{ss}
Oded Schramm and Jeffrey~E. Steif.
\newblock Quantitative noise sensitivity and exceptional times for percolation.
\newblock {\em Ann. of Math. (2)}, 171(2):619--672, 2010.

\bibitem{tal}
Michel Talagrand.
\newblock On {R}usso's approximate zero-one law.
\newblock {\em The Annals of Probability}, pages 1576--1587, 1994.

\bibitem{tassion}
Vincent Tassion.
\newblock Crossing probabilities for {V}oronoi percolation.
\newblock {\em Ann. Probab.}, 44(5):3385--3398, 2016.

\end{thebibliography}
\end{document}